\documentclass[12pt]{amsart}
\usepackage{amsfonts,amssymb}
\usepackage[all]{xypic}
\usepackage{amsmath}
\usepackage{amsthm,amscd,latexsym}
\usepackage{color}

\newtheorem{theorem}{Theorem}[section]
\newtheorem{corollary}[theorem]{Corollary}
\newtheorem{Definition}[theorem]{Definition}
\newtheorem{lemma}[theorem]{Lemma}
\newtheorem{proposition}[theorem]{Proposition}
\newtheorem{Example}[theorem]{Example}
\newtheorem{Remark}[theorem]{Remark}

\newenvironment{remark}{\begin{Remark}\begin{em}}{\end{em}\end{Remark}}
\newenvironment{example}{\begin{Example}\begin{em}}{\end{em}\end{Example}}
\newenvironment{definition}{\begin{Definition}\begin{em}}{\end{em}\end{Definition}}

\newcommand{\Bx}{{\mathbf x}}
\newcommand{\By}{{\mathbf y}}

\newcommand{\ve}{\varepsilon}

\DeclareMathOperator{\ua}{\uparrow\!}
\DeclareMathOperator{\da}{\downarrow\!}
\DeclareMathOperator{\Pro}{\mathcal{P}}

\begin{document}
\title[Ordered Probability Spaces]{Ordered Probability Spaces}
\author{Jimmie Lawson}

\address{Department of Mathematics, Louisiana State University,
Baton Rouge, LA70803, USA}\email{lawson@math.lsu.edu}


\maketitle
\begin{abstract}
Let $C$ be an open cone in a Banach space equipped with the Thompson metric with  closure a normal cone. 
The main result gives sufficient conditions for Borel probability measures $\mu,\nu$ on $C$ with finite first moment 
for which  $\mu\leq \nu$ in  the stochastic order induced by the cone to be order approximated by sequences $\{\mu_n\},\{\nu_n\}$ of uniform 
finitely supported measures in the sense that $\mu_n\leq \nu_n$ for each $n$ and $\mu_n\to \mu$, $\nu_n\to \nu$
in the Wasserstein metric.  This result is the crucial tool in developing a pathway for extending various inequalities on 
operator and matrix means, which include the harmonic, geometric, and arithmetic operator means on the cone
of positive elements of a $C^*$-algebra, to the space $\mathcal{P}^1(C)$ of Borel measures of finite first moment on $C$. As an
illustrative particular application, we obtain  the monotonicity of the Karcher geometric mean on $\mathcal{P}^1(\mathbb{A}^+)$ for the positive cone 
$\mathbb{A}^+$ of a $C^*$-algebra $\mathbb{A}$.

 \end{abstract}

\section{Introduction}  
The set $\mathcal{P}^1(M)$  of Borel probability measures with finite first moment on a metric space $M$ admits a standard metric called the Wasserstein metric.  
In \cite{St} K.-T. Sturm  considered contractive barycentric maps $\beta: \mathcal{P}^1(M)\to M$, maps that were metrically contractive and carried a point 
measure to the corresponding point.  If $M$ is complete, then a contractive barycentric map on the set $\mathcal{P}_0(M)$ of uniform, finitely
supported probability measures extends uniquely to a barycentric map on $\mathcal{P}^1(M)$, since $\mathcal{P}_0(M)$ is dense in $\mathcal{P}^1(M)$
in the topology of the Wasserstein metric.  In \cite{LL16} Y.\ Lim and the author have modified this result by giving appropriate conditions for uniformity 
 of an indexed family of symmetric means $\{G_n\}_{n\geq 1}$ to induce to a unique contractive barycentric map on $\mathcal{P}^1(M)$.

In this paper we revisit these results in the setting of metric spaces equipped with a closed partial order.  The partial order induces in a natural way
a partial order on $\mathcal{P}^1(M)$, the stochastic order. In our main result we give sufficient conditions to show if $\mu\leq \nu$ in $\mathcal{P}^1(M)$, then there exist
sequences $\{\mu_n\},\{\nu_n\}\subseteq \mathcal{P}_0(M)$ such that $\mu_n\to\mu$, $\nu_n\to \nu$, and $\mu_n\leq\mu_n$ for each $n$. 
(This result may be viewed as strengthening the result that $\mathcal{P}_0(M)$ is dense in $\mathcal{P}^1(M)$.)  

Our main motivation for deriving these results is the setting of the open cone of positive matrices with the Loewner order, or more generally the
cone of positive elements in a $C^*$-algebra.  We show in this case that the Thompson metric satisfies the hypotheses of our main theorem.
This allows us to conclude that if a contractive barycentric map $\beta:\mathcal{P}^1(M)\to M$ is order preserving when restricted to
$\mathcal{P}_0(M)$, then is order preserving on $\mathcal{P}^1(M)$.  It is frequently the case that order preservation is considerably easier 
to derive for $\mathcal{P}_0(M)$, so the paper provides an important and helpful pathway for showing order preservation of barycentric maps.

In related work  S. Kim and H. Lee \cite{KL} extended via the Bochner integral the theory of the widely studied 
least squares aka Cartan aka Karcher mean on the open cone $\mathbb{P}_n$ of positive definite $n\times n$-matrices to a barycentric map on the compactly supported Borel measures, with an extension to the square-integrable functions.   Using this approach they were able to extend basic properties of the Karcher mean
to this more general setting.  A more ambitious program was undertaken by M. P\'alfia \cite{Pa}, who generalized the techniques
of \cite{LL14} to define and approximate a barycentric map given by the solution of an integral  generalized Karcher equation for a given Borel
probability measure with bounded support on the cone $\mathbb{P}$ of positive operators on a Hilbert space.

\section{Preliminaries}
For a metric space $X$, let $\mathcal{B}(X)$ be the algebra of Borel sets, the smallest $\sigma$-algebra
containing the open sets.  A \emph{Borel measure} $\mu$ is a countably additive (positive) measure defined on $\mathcal{B}(X)$.
The support of $\mu$ consists of all points $x$ for which $\\mu(U)>0$ for each open set $U$ containing $x$.  The support of 
$\mu$ is always a closed set.  The \emph{finitely supported measures} are those of the form $\sum_{i=1}^n r_i\delta_{x_i}$, where
for each $i$, $r_i\geq 0$, $\sum_{I=1}^n r_i=1$, and $\delta_{x_i}$ is the point measure of mass $1$ at the point $x_i$.
As far as the author has determined, the earliest reference for the following result is \cite{MS}.
\begin{lemma}\label{L:sep}
Let $X$ be a metric space and $\mu$ a finite Borel measure, i.e., one for which  $\mu(X)<\infty$.  Then supp$(\mu)$,  the support of $\mu$, is separable.
\end{lemma}

\begin{proof} The case that supp$(\mu)$  has cardinality less than two is trivial, so we assume there exist $a,b\in$ supp$(\mu)$ such that $a\ne b$.
Let $d(a,b)\geq 1/N$, $N$ a positive integer.  Let us call a nonempty subset $A$ of  $X$  $\ve-$\emph{scattered} if $d(x,y)\geq \ve$ whenever
$x,y\in A$ and $x\ne y$.   For $m\geq N$, let $\mathcal{A}_m=\{A\subseteq\,\textrm{supp}(\mu): A \mbox{ is } (1/m)-\textrm{scattered}\}$.
Note that $\{a,b\}$ is $1/m$-scattered and the union of any collection of members of $\mathcal{A}_m$ totally ordered by inclusion is again
a member of $\mathcal{A}_m$.  Hence by Zorn's Lemma, there exists a maximal member $A_m$ of $\mathcal{A}_m$.  By the maximality of
$A_m$, if $x\in\,$supp$(\mu)$, then $d(x,a)<1/m$ for some $a\in A_m$ (otherwise we can make $A_m$ larger by adding $x$).

We claim that $A_m$ is countable.  First note that since $A_m$ is $1/m$-scattered, the collection $\{B_{1/2m}(a): a\in A_m\}$
of open balls of radius $1/2m$ is pairwise disjoint. For each $k\in \mathbb{N}$, define $A_{m,k}=\{a\in A_m :
\mu(B_{1/2m}(a))\geq 1/k\}$. Note that for $n$ members of $A_{m,k}$, the union $\bigcup_{i=1}^n B_{1/2m}(a_i)$ has
measure at least $n/k\leq \mu(X)$. Thus $n$ is bounded above by $k\mu(X)$,  from which it follows that $\mathcal{A}_{m,k}$ is finite.
Note for every $a\in A_m$, $\mu(B_{1/2m}(a))>0$ since $a\in\,$supp$(\mu)$.  It follows that $A_m=\bigcup_kA_{m,k}$, a countable union
of finite sets, and hence $A_m$ is countable.  Thus $D=\bigcup_m A_m$ is a countable union of countable sets, hence countable.
The fact that for any $x\in\,$supp$(\mu)$ and any $m$, $d(x,A_m)<1/m$ implies that $D$ is dense.
 \end{proof}
 
 \begin{remark}\label{R:suppeqone}
 Let $\mu$ be a finite Borel measure on $X$, a metric space. \\
(i) For $X$ separable, the complement of the support is a countable union of open sets of measure $0$, hence has measure $0$,
and thus $\mu(\mathrm{supp}\,\mu)=1$.\\
(ii)  Similarly for $\tau$-additive measures one may realize the complement of the support as a directed union of open sets of measure $0$ 
 hence the complement has measure $0$, and the support has measure $1$. (Recall that a Borel measure $\mu$ is $\tau$-\emph{additive}
 if $\mu(\bigcup_\alpha U_\alpha)=\sup_\alpha \mu(U_\alpha)$ for all directed families $\{U_\alpha: \alpha\in D\}$
 of open sets.)
 \end{remark}

 In what follows we wish to avoid dealing with the problematic and pathological case of probability 
 measures that are not fully supported in the sense that the measure of their support is less than one.
We can give a convenient characterization of the measures we will consider.  First we recall the \emph{Prohorov metric} $\pi(\mu,\nu)$
defined for two Borel probability measures $\mu,\nu$ on $X$ as the infimum of all $\ve>0$ such that for all closed sets $A$,
$$\mu(A)\leq \nu(A^\ve)+\ve,~~~\nu(A)\leq \mu(A^\ve)+\ve,$$
where $A^\ve=\{x\in X:d(x,y)<\ve \mbox{ for some }y\in A\}$.  

\begin{proposition}\label{P:support1}
 For  $\mu$ a Borel probability on a metric space $(X,d)$, the following are equivalent.\\
(1) There exists a sequence $\{\mu_n\}$ of finitely supported probability measures $($with rational coefficients$)$ that converges to $\mu$ with respect to
the Prohorov metric.\\
(2) The support of $\mu$ has measure $1$.\\
(3) The measure $\mu$ is $\tau$-additive.
\end{proposition}

\begin{proof}  ($1\Rightarrow 2$)  By (iii) of the section on the Prohorov metric of \cite{Bi}, $\{\mu_n\}$ converges weakly to $\mu$.
 Let $A$ be the closure of $\bigcup_{n=1}^\infty \mathrm{supp},\mu_n$.  Then $A$ is a separable subspace of  $X$.  We note
 that $\mu_n(A)=1$ for each $n$.   Then $\mu(A)=1$ by the Portmanteau Theorem \cite[Theorem 2.1(iii)]{Bi}. It follows that the
 complement of $A$ has $\mu$-measure $0$ (hence misses the support) and that the
 restriction of $\mu$ to $A$ is a Borel probability measure on the separable space $A$. Hence by Remark \ref{R:suppeqone}(i) 
 the support of $\mu$, which will be the same for $\mu$ and $\mu$ restricted to $A$, has measure $1$.  
 
($2\Rightarrow 1$)  The measure $\mu$ restricted to $A=$supp$(\mu)$  is a Borel probability measure on $A$, which is separable
by Lemma \ref{L:sep}.   It is standard fact that the space of Borel probabilty measures on a separable metric space 
endowed with the Prohorov metric is again
separable, and the proof of this fact typically involves constructing a sequence of finitely supported measures $\{\mu_n\}$ 
with rational coefficients that are dense (see item (vi) leading up to Theorem 6.8 of \cite{Bi}).  Applying this fact to $A$ yields the desired sequence.

($3\Leftrightarrow 2$) One direction follows from Remark \ref{R:suppeqone}.  Assume (2) and note for any Borel set $B$, 
$$\mu(B)=\mu(B\cap\textrm{supp}(\mu))+\mu(B\setminus \textrm{supp}(\mu))=\mu(B\cap\textrm{supp}(\mu))+0=\mu(B\cap\textrm{supp}(\mu)).$$
Let $\{U_\alpha: \alpha\in D\}$ be a directed
family of open sets with union $W$. Then in the subspace supp$(\mu)$, $\{U\cap\textrm{supp}(\mu): U\in D\}$ is a directed family of open sets
with union $W\cap U$.  It is standard that every Borel probability measure in a separable metric space is $\tau$-additive (one can pick a countable
subset of $D$ with the same union), so 
$$\mu(W)=\mu(W\cap \textrm{supp}(\mu))=\sup\{\mu(U\cap\textrm{supp}(\mu):U\in D\}=\sup\{\mu(U): U\in D\}.$$
\end{proof}

Let ${\mathcal P}(X)$ be the set of all fully supported Borel probability measures on $(X, {\mathcal B}(X))$ (in the sense that
$\mu($supp$(\mu))$=1)
and  ${\mathcal P}_{0}(X)$ the set of all $\mu\in {\mathcal P}(X)$ of the form
$\mu=\frac{1}{n}\sum_{j=1}^{n}\delta_{x_{j}}$ with  $n\in {\mathbb N},$ where $\delta_x$ is the point measure of mass $1$ at $x$.
Members of $\mathcal{P}_0(X)$ are also referred to as uniform probability measures with finite support.
For $p\in[1,\infty)$ let ${\mathcal P}^{p}(X)\subseteq{\mathcal P}(X)$ be the set of probability measures with \emph{finite $p$-moment}: for some (and hence all) $y\in X,$
$$\int_{X}d^p(x,y)d\mu(x)<\infty.$$
For $p=\infty$, $\mathcal{P}^\infty(X)$ denotes the set of probability measures with bounded support.  We sometimes
refer to members of $\mathcal{P}^1(X)$ as \emph{integrable} probability measures.

For metric spaces $X$ and $Y$, a map $f:X\to Y$ is \emph{measurable}  if $f^{-1}(A)\in\mathcal{B}(X)$  whenever $A\in\mathcal{B}(Y)$.
For $f$ to be measurable, it suffices that $f^{-1}(U)\in\mathcal{B}(X)$ for each open subset $U$ of $Y$, and hence continuous functions
are measurable.  A measurable map $f:X\to Y$  induces a \emph{push-forward} map $f_*:\Pro(X)\to\Pro(Y)$ defined by
$f_*(\mu)(B)=\mu(f^{-1}(B))$ for $\mu\in\Pro(X)$ and $B\in\mathcal{B}(Y)$.  Note that $\mathrm{supp}(f_*(\mu))=f(\mathrm{supp}(\mu))^-$,
the closure of the image of the support of $\mu$.

We say that $\omega\in\Pro(X\times X)$ is a \emph{coupling} for $\mu,\nu\in\Pro(X)$ and that $\mu,\nu$ are \emph{marginals} for
$\omega$ if for all $B\in\mathcal{B}(X)$
\begin{equation*}
\omega(B\times X)=\mu(B) \qquad \mathrm{and} \qquad \omega(X\times B)=\nu(B).
\end{equation*}
Equivalently $\mu$ and $\nu$ are the push-forwards of $\omega$ under the projection maps $\pi_1$ and $\pi_2$ resp.  We note that
one such coupling is the product measure $\mu\times \nu$, and that for any coupling $\omega$ it must be the case that
$\mathrm{supp}(\omega)\subseteq \mathrm{supp}(\mu)\times\mathrm{supp}(\nu)$.  We denote the set of all couplings for $\mu,\nu \in \Pro(X)$ by
$\Pi(\mu,\nu)$.

For $1\leq p<\infty$, the  $p$-Wasserstein distance  $d_p^W$ (alternatively Kantorovich-Rubinstein distance)
on ${\mathcal P}^{p}(X)$ is defined by
$$d_p^W(\mu_{1},\mu_{2}):=\bigg(\inf_{\pi\in\Pi(\mu_1,\mu_2)} \int_{X\times X} d^p(x,y) d\pi(x,y)\bigg)^{1/p}.$$
 It is known that $d_p^W$  is a complete metric on ${\mathcal P}^{p}(X)$ whenever $X$ is a complete metric space
and ${\mathcal P}_{0}(X)$ is $d_p^W$-dense in ${\mathcal P}^{p}(X)$ \cite{BO,St}.  Furthermore, it follows from the H\"older
inequality that $d_p^W\leq d_{p'}^W$ whenever $p\leq p'$.  We will be working almost exclusively with $d_1^W$, which we
henceforth write more simply as $d^W$.

The space $\Pro(X)$ is convex in the sense that $(1-t)\mu+t\nu\in\Pro(X)$ whenever $\mu,\nu\in\Pro(X)$.  It is easy to see that
$\Pro^1(X)$ is a convex subset of $\Pro(X)$.
\begin{lemma}\label{L:convex}
For $\mu_1,\mu_2,\nu\in\Pro^1(X)$, $d^W((1-t)\mu_1+t\mu_2, \nu)\leq (1-t)d^W(\mu_1,\nu)+td^W(\mu_2,\nu)$.
\end{lemma}

\begin{proof} If $\omega_1\in\Pi(\mu_1,\nu)$ and $\omega_2\in\Pi(\mu_2,\nu)$, then it is straightforward to
see that $(1-t)\omega_1+t\omega_2\in\Pi((1-t)\mu_1+t\mu_2,\nu)$. Thus
\begin{eqnarray*}
d^W((1-t)\mu_1+t\mu_2,\nu)&\leq& \int_{X\times X} d(x,y)\,d((1-t)\omega_1+t\omega_2)\\
&=&(1-t)\int_{X\times X} d(x,y) d\omega_1+t\int_{X\times X}d(x,y) d\omega_2.
\end{eqnarray*}
Taking infs over $\omega_1\in\Pi(\mu_1,\nu)$ and $\omega_2\in\Pi(\mu_2,\nu)$ yields the
desired inequality.
\end{proof}

\section{Approximating Probability Measures with Bounded Ones}
In this section we present for an integrable  probability measure $\mu$ in a metric space $M$ a scheme for approximating $\mu$ 
arbitrarily close in the Wasserstein metric by probability measures with bounded support.  
\begin{lemma}\label{L:upbound}
Let $M$ be a metric space and let $f:M\to M$ be a Borel measurable map.  If $q\in \mathcal{P}^1(M)$
and $\int_M d(x,f(x)) dq<\infty$, then $f_*(q)\in \Pro^1(M)$ and
$d^W(q,f_*(q))\leq \int_M d(x,f(x)) dq$.
\end{lemma}

\begin{proof}  We first note by the change of variables formula that
$$\int_M d(x,z) df_*(q)(x)=\int_M d(f(x),z) dq(x)\leq \int_M \left(d(x,f(x))+d(x,z)\right)dq(x)<\infty,$$
so $f_*(q)\in\Pro^1(M)$.

Next consider the measurable map $F:M\to M\times M$ defined by $F(x)=(x,f(x))$, and let
$\mu=F_*(q)\in\mathcal{P}(M\times M)$.
For $\pi_1: M\times M\to M$, projection into the first coordinate, $(\pi_1)_*(\mu)=(\pi_1 F)_*(q)=(1_M)_*(q)=q$.
For the second projection $\pi_2$, $(\pi_2)_*(\mu)=(\pi_2F)_*(q)=f_*(q)$.   So $\mu$ has marginals $q$ and $f_*(q)$.
By definition
$$d^W(q,f_*(q))\leq \int_{M\times M} d(x,y) d\mu=\int_{M\times M} d(x,y)d(F_*(q))=\int_M d(x,f(x)) dq,$$
where the last equality is just the change of variables formula.
\end{proof}

\begin{lemma}\label{L:a1}
Let $M$ be a metric space, $a\in M$, and let $q\in \mathcal{P}^1(M)$.  Let $A_n$ be an increasing sequence of Borel sets such that
$M=\bigcup_n A_n$.  Define $f_n: M\to M$ by $f_n(x)=x$ if $x\in A_n$ and $f_n(x)=a$ otherwise.  Then $\lim_n (f_n)_*(q)= q$ in the Wasserstein
metric.
\end{lemma}

\begin{proof}
Let $\ve >0$. Since $q\in \mathcal{P}^1(M)$,  $\int_M d(x,a)\, dq(x)<\infty$.  Let $\chi_n$ be the characteristic function for  $A_n$.
By Lebesgue's Dominated Convergence Theorem
\begin{eqnarray*}
\lim_n \int_{A_n} d(x,a)\, dq(x)&=&\lim_n\int_M \chi_n(x)d(x,a)\,dq(x)=\int_M d(x,a)\,dq(x)<\infty, \mbox{ so}\\
\lim_n\int_{M\setminus A_n} d(x,a)\,dq(x)&=&\lim_n \left(\int_M d(x,a)\,dq(x)-\int_{A_n}  d(x,a)\,dq(x)\right)= 0.
\end{eqnarray*}
Thus $\int_{M\setminus A_n} d(x,a)\, dq(x)<\ve$ for all $n$
large enough.  Then
$$\int_M d(x,f_n(x)) dq=\int_{A_n} d(x,x)\,dq+\int_{M\setminus A_n} d(x,a)\, dq(x)<0+\ve=\ve$$
for large $n$.  Then $(f_n)_*(q)\in\Pro^1(M)$ and $\lim_n (f_n)_*(q)= q$
now follow from Lemma \ref{L:upbound}.
\end{proof}

We need to have the freedom to send the complement of $A_n$ into a varying point, but this requires an extra hypothesis.
\begin{lemma}\label{L:a2}
Let $M$ be a metric space, $a\in M$, and let $q\in \mathcal{P}^1(M)$. Let $A_n$ be an increasing sequence of Borel sets such that
$M=\bigcup_n A_n$ and let $\{a_n\}$ be a sequence in $M$.
Define $f_n, g_n: M\to M$ by $f_n(x)=x=g_n(x)$ if $x\in A_n$ and $f_n(x)=a$, $g_n(x)=a_n$ otherwise.
If there exists $K>0$ such that for each $n$, there exists an open
ball $B_{r_n}(a)\subseteq A_n$ such that $Kr_n\geq d(a,a_n)$,
then $\lim_n (g_n)_*(q)= q$  in the Wasserstein
metric and $\lim_n d^W((g_n)_*(q),(f_n)_*(q))=0$.
\end{lemma}

\begin{proof} By Lemma \ref{L:a1} $(f_n)_*(q)\in\Pro^1(M)$.
Note that by hypothesis $M\setminus A_n\subseteq M\setminus B_{r_n}(a)$,
so $\int_{M\setminus A_n} d(x,a) dq(x)\geq  \int_{M\setminus A_n} r_n\, dq=r_n q(M\setminus A_n)$.
Since $q\in\mathcal{P}^1(M)$, we may choose $A_n$ large enough so
that $\int_{M\setminus A_n} d(x,a)dq(x)<\ve/K$ (see the proof of the previous lemma).
We conclude that $q(M\setminus A_n)< \ve/Kr_n$, so that
$d(a_n,a)q(M\setminus A_n)<Kr_n(\ve/Kr_n)=\ve$. We thus have
\begin{eqnarray*}
\int_M d(x,g_n(x))dq &\leq & \int_{A_n}d(x,x) dq +\int_{M\setminus A_n} d(x,a_n) dq(x)\\
&\leq & 0+\int_{M\setminus A_n} d(x,a) dq(x)+ \int_{M\setminus A_n} d(a,a_n) dq \\
&\leq & \int_{M\setminus A_n} d(x,a) dq(x)+ d(a,a_n) q(M\setminus A_n). \end{eqnarray*}
As $n\to \infty$, the first term goes to $0$ by the proof of the preceding lemma and the
second term goes to $0$ by the earlier part of this proof.  That
$\lim_n (g_n)_*(q)= q$, i.e., $\lim_n d^W(q, (g_n)_*(q))=0$, now follows from Lemma \ref{L:upbound}.
Since also $ \lim_n (f_n)_*(q)=q$ by the preceding lemma, we conclude
$\lim_n d^W((g_n)_*(q),(f_n)_*(q))=0$.
\end{proof}

\section{Approximation in Ordered Spaces}
\begin{definition}
An \emph{ordered metric space} is a metric space $M=(M,d,\leq)$ equipped with a partial order
that is \emph{closed} in the sense that $\{(x,y):x\leq y\}$ is a closed subset of $M\times M$.
\end{definition}
\emph{Throughout this section $M=(M,d,\leq)$ will denote an ordered metric space.}
For
$A\subseteq M$, we define $\ua A=\{y\in M: x\leq y \mbox{ for some } x\in A\}$.  A set $A$ is an
\emph{upper set} if $\ua A=A$.  Lower sets and $\da A$ are defined in a dual fashion (with respect
to the order).  We write $\ua x$ resp.\ $\da x$ for $\ua\{x\}$ resp.\ $\da\{x\}$. An \emph{order
interval} is a set of the form $[x,y]:=\{w\in M: x\leq w\leq y\}=\ua x\cap\da y$.  Having a closed partial
order implies $\ua x$, $\da x$, and $[x,y]$ are all closed sets.

\begin{definition}
For $p,q\in \mathcal{P}(M)$, we define $p\leq q$ if $p(U)\leq q(U)$ for all open upper sets $U$.
\end{definition}
This order is sometimes referred to as the \emph{stochastic} order. The stochastic order simplifies for finitely supported
measures.
\begin{lemma}\label{L:stord}
Suppose that $X$ is a partially ordered topological space for which $\da x$ is closed for eachd $x\in X$.  For finitely supported
measures $\mu,\nu$, we have $\mu\leq \nu$ if and only if $\mu(A)\leq \nu(\ua A)$ for each $A\subseteq$ supp$(\mu)$.
\end{lemma}

\begin{proof}  Suppose $\mu\leq \nu$ and let $A\subseteq$ supp$(\mu)$.  Let $U$ be the complement of 
$\bigcup\{\da x:x\in\mbox{supp}(\nu)\setminus \ua A\}$.  By hypothesis the finite union is closed and hence $U$ is
open. It is the complement of a lower set, hence an upper set.
Since $\da x\cap A=\emptyset$ for $x\notin \ua A$, we conclude $U$ contains $A$.  Hence 
$$\mu(A)\leq \mu(U)\leq \nu(U)=\nu(U\cap\mbox{supp}(\nu))=\nu(\ua A\cap\mbox{supp}(\nu))=\nu(\ua A).$$ 
Conversely suppose $\mu(A)\leq \nu(\ua A)$ for each $A\subseteq$ supp$(\mu)$. Let $U=\ua U$ be an open
upper set.  Set $A=U\cap\mbox{supp}(\mu)$.  Then  
$\mu(U)=\mu(A)\leq \nu(\ua A)\leq \nu(U)$.
\end{proof}

\begin{proposition}\label{P:ordapprox}
Let  $p,q\in \mathcal{P}^1(M)$ such that $q\leq p$, where $M$ is an ordered metric space.
Let $A_n=[z_n,w_n]$ be an increasing sequence  of order intervals such that
$M=\bigcup_n A_n$. Suppose for some $a\in M$ there exists $K>0$ and a sequence $\{r_n\}$ of positive numbers such that $B_{r_n}(a)\subseteq A_n$
and $d(a,z_n),d(a,w_n)\leq Kr_n$ for each $n$.  Define $f_n$ and $g_n$ by $f_n(x)=x=g_n(x)$ for $x\in A_n$ and $f_n(x)=z_n$,
$g_n(x)=w_n$ otherwise.  Then $(f_n)_*(q)\leq (g_n)_*(p)$ for each $n$ and $\lim_n (f_n)_*(q)=q$ and $\lim_n (g_n)_*(p)=p$.
\end{proposition}

\begin{proof} Define $F_n:M\to M$ by $F_n(x)= x$ for $x\in A_n$ and $F_n(x)=a$ otherwise.  By Lemma \ref{L:a2}
$\lim_n d^W((g_n)_*(p),(F_n)_*(p))=0$ and by Lemma \ref {L:a1} $\lim_n (F_n)_*(p)=p$. From these two assertions
we conclude that $\lim_n (g_n)_*(p)=p$.  Similarly $\lim_n (F_n)_*(q)=q$ and hence $\lim_n (f_n)_*(q)=q$.

Let $U$ be an open upper set.  By hypothesis $q(U)\leq p(U)$.   For fixed $n$, if $U$ contains $z_{n}$, then
$(f_n)_*(q)(U)=1=(g_n)_*(p)(U)$ and if $U$ misses $[z_n,w_n]$ then $(f_n)_*(q)(U)=0=(g_n)_*(p)(U)$.
Assume $U$ hits $[z_n,w_n]$, but $z_n\notin U$.  Since all the $q$-mass outside $[z_n,w_n]$ is stored
at $z_n$ by $(f_n)_*(q)$, we conclude $(f_n)_*(q)(U)=q(U\cap[z_n,w_n])\leq q(U)$.  Since all the $p$-mass
outside $[z_n,w_n]$ is stored at $w_n$ and $w_n\in U$, we conclude that
\begin{eqnarray*}
(g_n)_*(p)(U)&=&p(U\cap [z_n,w_n])+p(M\setminus [z_n,w_n])\\
&\geq& p(U\cap [z_n,w_n]) +p(U\cap (M\setminus [z_n,w_n]))=p(U).\end{eqnarray*}
We conclude $(f_n)_*(q)\leq (g_n)_*(p)$.
\end{proof}

\begin{lemma}\label{L:b0}
Let $M$ be an ordered metric space and let $q=\sum_{i=1}^n w_i\delta_{x_i}$ be a finitely supported measure.  If $\ve >0$ and $z\leq x_i$ for $1\leq i\leq n$, then
there exists $p=(1/N)\sum_{i=1}^N \delta_{y_i}$  such that $d^W(q,p)<\ve$ and $p\leq q$.  Furthermore, $y_i\in\{z,x_1,\ldots, x_n\}$ for $1\leq i\leq N$.
\end{lemma}

\begin{proof}  Pick $B>0$ such that $d(x_i,z)\leq B$ for $1\leq i\leq n$.  For each $i$, $1\leq i\leq n$, choose a dyadic rational $r_i=m_i/2^{n_i}\in \mathbb{Q}$
such that $0\leq w_i-r_i<\ve/nB$.  Define $\omega$ on $M\times M$ by $\omega(\{(x_i,x_i)\})=r_i$ and  $\omega(\{(x_i,z)\})= w_i-r_i$ for $1\leq i\leq n$ and
$\omega(\{(x,y)\})=0$ otherwise. Then for $1\leq i\leq n$, $w_i=\omega(\{x_i\}\times M)$ and $0=\omega(\{z\}\times M)$, so the first marginal of
 $\omega$ is $q$.  Further $r_i=\omega(M\times \{x_i\})$ and $1-\sum_{i=1}^n r_i=\omega(M\times\{z\})$.  This gives the marginal $p$, where
 $p(\{x_i\})=r_i$ and $p(\{z\})=1-\sum_{i=1}^n r_i$, and $p(\{x\}) = 0$ for $x \in M \setminus \{z, x_1,\ldots, x_n\}$.  We have
 \begin{eqnarray*}
 d^W(q,p)\leq \int_{M\times M} d(x,y) d\omega=\sum_{i=1}^n d(x_i,x_i) r_i+d(x_i,z)(w_i-r_i)
 < 0+Bn\frac{\ve}{nB}=\ve. \end{eqnarray*}
Since in passing from $q$ to $p$ we have lowered the weight at each $x_i$ and placed the extra weight at the bottom point $z$,
it follows easily that $p(U)\leq q(U)$ for any open upper set, and hence $p\leq q$.

Let $N=\max\{ 2^{n_i}: i=1,\ldots, n\}$.  Then each $r_i=m_i/2^{n_i}=k_i/N$ for some $0\leq k_i\leq N$.
We obtain   $p=(1/N)\sum_{k=1}^N \delta_{y_k}$ where $y_k=x_i$ appears $k_i$ times for each $i$ and
$z$ appears $N-\sum_{i=1}^n k_i$ times.
\end{proof}

\begin{definition}
A partial ordered space $M$  is said to have a basis of neighborhoods consisting of order intervals if given any
open set $U$ and $x\in U$, there exist an open set $V$ and an order interval $[z,w]$ such that $x\in V\subseteq [z,w]
\subseteq U$.
\end{definition}

\begin{theorem}\label{T:b1}
Let $M$ be a metric space equipped with a closed order and having a basis of neighborhoods consisting
of order intervals.  Let $q\in\mathcal{P}^1(M)$ with support supp$(q)$
metrically bounded and contained in the interior of $[z,w]$, an order interval.
Then there  exists a sequence $\{q_n\}\subseteq \Pro_0(M)$ of finitely supported uniform measures with support contained in $[z,w]$
such that in the Wasserstein space $\Pro^1(M)$,  $\lim_n q_n=q$ and for each $n$, $q_n\leq q$ (alternatively $q_n\geq q)$.
\end{theorem}

\begin{proof}   The case $q$ is a point mass is trivial (let $q_n=q$), so we assume that is not the case.
Let $(z,w)$  denote the interior of $[z,w]$, the largest open set contained in $[z,w]$.
We construct the desired sequence in two steps.\\
\emph{Step 1}.  Let $\ve>0$.  For each $x\in\mathrm{supp}(q)$, pick an order interval $[z_x,w_x]$ and an open set $V_x$  such that
$x\in V_x\subseteq [z_x,w_x]\subseteq B_r(x)\cap (z,w)$, where $r=\ve/4$.  Since supp$(q)$ is separable by Lemma \ref{L:sep},
hence second countable, countably many
of the $\{V_x\}$, say $\{V_i\}_{i\in\mathbb N}$, cover supp$(q)$.  Let $[z_i,w_i]$ denote the corresponding order interval containing
$V_i$.

Define a sequence of sets $A_i$ inductively by $A_1=V_1$, $A_{k+1}=V_{k+1}\setminus \bigcup_{i=1}^k V_i$.  We delete those
$A_k$ that are empty, and renumber accordingly.  The collection $\{A_k\}$ forms a partition  of $\bigcup_i V_i$.  In particular,
$\bigcup_k A_k$ contains supp$(q)$, and hence has measure $1$ (by our definition of $\mathcal{P}(M)$). 
Since the collection $\{A_k\}$ is a partition, it follows that
$1=q(\bigcup_k A_k)=\sum_{k=1}^\infty q(A_k)$.   Since the diameter of supp$(q)$ is assumed finite, we can pick $b>\ve$ that is
greater than the diameter of $\{z\}\cup \mathrm{supp}(q)$.  We can also  pick $m$ large enough so that
$\sum_{k=1}^m q(A_k)>1-\ve/2b$, which means $q(\mathrm{supp}(q)\setminus \bigcup_{k=1}^m A_k)<\ve/2b$.
Denote $\mathrm{supp}(q)\setminus \bigcup_{k=1}^m A_k$ by $D$.

 We define $f:M\to M$ by $f(x)=z_i$ if $x\in A_i$ for $1\leq i\leq m$  and $f(x)=z$ otherwise.
Then $f(x)\leq x$ for each $x\in[z,w]$ and $f$ is Borel measurable since $f(M)$ is finite
and the inverse image of each point is a Borel set.   We note also for $x$ in $A_k$ for $1\leq k\leq m$ that
$d(x,f(x))=d(x, z_k)<\ve/2$ since each $[z_k,w_k]$ is contained in an open ball of radius $\ve/4$. We compute
\begin{eqnarray*}
\int_M d(x,f(x))dq &=& \sum_{k=1}^m\int _{A_k} d(x,f(x)) dq+\int_D d(x,f(x)) dq\\
&\leq & \sum_{k=1}^m \frac{\ve}{2} q(A_k)+b q(D)\\
&\leq & \frac{\ve}{2} +b\frac{\ve}{2b}=\ve.
\end{eqnarray*}
We conclude from Lemma \ref{L:upbound} that $d^W(q, f_*(q))<\ve$.

For each $U=\ua U$ open in $M$, $f^{-1}(U)\cap [z,w]\subseteq U$ since $f(x)\leq x$ for $x\in [z,w]$.  We thus have
$$f_*(q)(U)=q[f^{-1}(U)] =q(f^{-1}(U)\cap [z,w])\leq q(U),$$
where the second equality holds since supp$(q)\subseteq [z,w]$.  Thus $f_*(q)\leq q$.

\emph{Step 2}.  We construct as in  step 1 for $n\in\mathbb{N}$ and $\ve=1/n$ a function $f_n:M\to M$ with finite image
contained in $[z,w]$ such that  $d^W(q, (f_n)_*q)<1/n$, $(f_n)_*(q)\leq q$.  Thus the sequence $(f_n)_*(q)$ converges to $q$ in
$\Pro^1(M)$ equipped with the Wasserstein metric.   Since $(f_n)_*(q)$ is  finitely supported measure, by Lemma
\ref{L:b0} we may pick $q_n\in\Pro_0(M)$ such that $q_n\leq (f_n)_*(q)$ and $d^W(q_n,(f_n)_*(q))<1/n$.  Then it must be the case
that $\lim_n q_n=q$ since $\lim_n (f_n)_*(q)= q$.  We have further for each $n$ that $q_n\leq (f_n)_*(q)\leq q$.  Lemma
\ref{L:b0} allows us to assume that supp$(q_n)\subset$ supp$(f_n)_*(q)\subseteq [z,w]$.

\end{proof}

We come now to our main approximation theorem for ordered spaces.
\begin{theorem}\label{T:ordpair}
Let $M$ be a metric space equipped with a closed order and having a basis of neighborhoods consisting
of order intervals. Let $A_n=[z_n,w_n]$ be an increasing sequence  of metrically bounded
order intervals such that  $A_n$ is contained
in the interior of $A_{n+1}$ for each $n$ and such that
$M=\bigcup_n A_n$. Suppose for some $a\in M$ there exists $K>0$ and a sequence $\{r_n\}$ of positive numbers
such that $B_{r_n}(a)\subseteq A_n$  and $d(a,z_n),d(a,w_n)\leq Kr_n$ for each $n$. Then for $p,q\in \mathcal{P}^1(M)$
with $q\leq p$, there exist sequences $\{q_n\},\{p_n\}\subseteq \Pro_0(M)$ such that
$q_n\to q$ and $p_n\to p$ in the Wasserstein distance and $q_n\leq p_n$ for each $n$.
\end{theorem}

\begin{proof}
By Proposition \ref{P:ordapprox} there exists sequences $(f_n)_*(q)\to q$ and $(g_n)_*(p)\to p$ in the Wasserstein space
$\mathcal{P}^1(M)$ such that   $(f_n)_*(q)\leq (g_n)_*(p)$ for each $n$.  By definition $(f_n)_*(q)$ and $(g_n)_*(p)$ have
 support contained in the order interval $[z_n,w_n]$ and hence the support is metrically bounded and contained
 in the interior of $[z_{n+1},w_{n+1}]$.
 By Theorem \ref{T:b1} and its order dual we may for each $n$ pick $q_n,p_n\in\Pro_0(M)$
 with support contained in $[z_{n+1},w_{n+1}]$ so that
 $$d^W(q_n,(f_n)_*(q))<d^W((f_n)_*(q),q)+1/n,~~~ d^W(p_n,(g_n)_*(p))< d^W((g_n)_*(p), p)+1/n$$
and  $q_n\leq (f_n)_*(q)\leq (g_n)_*(p)\leq p_n$
for each $n$.  Since $(f_n)_*(q)\to q$ and $(g_n)_*(p)\to p$,
it follows that $p_n\to p$ and $q_n\to q$, yielding the desired sequences.
\end{proof}

\section{The Thompson Metric}
In this section we consider an important metric setting where the conditions of Theorem \ref{T:ordpair} are 
satisfied, an example that  served as an important motivation for the previous work.  
Let $V$ be a real Banach space and let $\Omega$ denote
 a non-empty open
convex cone of $V$: $t\Omega\subset \Omega$ for all $t>0,$
$\Omega+\Omega\subset \Omega$, and ${\overline \Omega}\cap
-{\overline \Omega}=\{0\},$ where $\overline \Omega$ denotes the
closure of $\Omega.$   We consider the partial order on $V$ defined
by $$x\leq y \ \ {\mathrm{if\ and\ only\ if}}\ y-x\in {\overline
\Omega}.$$ We write $x<y$ if $y-x\in\Omega.$ We further assume that
$\Omega$ is a \emph{normal} cone: there exists a constant $K$ with
$||x||\leq K||y||$ for all $x,y\in \Omega$ with $x\leq y.$ For
$x\leq y,$ we denote by $[x,y]$ the closed order interval
$$[x,y]:=\{z\in V: x\leq z\leq y\}$$ and the open order interval
by  $(x,y)=\{z\in V: x<z<y\}$ whenever $x<y.$

Any member $a$ of $\Omega$ is an order unit for the ordered space
$(V,\leq)$, and hence $|x|_{a}:=\inf\{\lambda>0: -\lambda a\leq
x\leq \lambda a\}$ defines a norm. By Proposition 1.1 in
\cite{Nu94}, for a normal cone $\overline \Omega$, the order unit
norm $|\cdot|_{a}$ is equivalent to $||\cdot||.$

A. C. Thompson \cite{Th} (cf. \cite{Nu}, \cite{Nu94}) has proved
that $\Omega$ is a complete metric space with respect to the
\emph{Thompson part metric} defined by
$$d(x,y)={\mathrm{max}}\{\log M(x/y), \log M(y/x)\}$$ where
$M(x/y):={\mathrm{inf}}\{\lambda>0: x\leq \lambda y\}=|x|_{y}.$
(Note that $x\leq M(x/y)y$ by the closedness of the relation $\leq$.)
Furthermore, the topology induced by the Thompson metric
agrees with the relative Banach space topology.

\begin{remark}\label{R:Tballs}
First we note for $a\in\Omega$ and $r>0$ that
$$d(a,x)\leq r\Leftrightarrow \,a\leq e^r x,\, x\leq e^r a\Leftrightarrow e^{-r}a\leq x\leq e^r a,$$
so the closed ball  around $a$ of radius $r$ (in the Thompson metric)
is the closed order interval $[e^{-r}a,e^ra]$.  Since the open
ball $B_r(a)$ around $a$ of radius $r$ is is the union of all closed balls of strictly smaller
radii, we have $B_r(a)=\bigcup_{0<t<r} [e^{-t}a,e^ta]$.

\end{remark}

\begin{proposition}\label{P:Thomp}
Let $V$ be a real  Banach space and let $\Omega$ be
a non-empty open normal cone equipped with the Thompson metric.
For  $p,q\in \mathcal{P}^1(\Omega)$
with $q\leq p$, there exist sequences $\{q_n\},\{p_n\}\subseteq \Pro_0(\Omega)$ such that
$q_n\to q$ and $p_n\to p$ in the Wasserstein distance and $q_n\leq p_n$ for each $n$.
\end{proposition}

\begin{proof}  Let $a\in\Omega$. We verify the hypotheses of Theorem \ref{T:ordpair} are met.
We let $A_n=[e^{-n}a,e^{n}a]$, the closed ball of radius $n$ around $a$. By the preceding
remark $A_n\subseteq  B_{n+1}(a)\subseteq A_{n+1}$.  Since $a$ is an order unit for $\Omega$,
$\Omega=\bigcup_n A_n$.  Choosing $r_n=n$  and $K=1$, we see by the preceding remark that
the conditions of Theorem \ref{T:ordpair} are satisfied, and hence the proposition follows.
\end{proof}

\begin{example}\label{E:Cstar}
Let ${\mathbb A}$ be a unital $C^{*}$-algebra with
identity $e,$ and let ${\mathbb A}^{+}$ be the set of positive
invertible elements of ${\mathbb A}.$  It is standard that
 ${\mathbb A}^{+}$ is a normal open cone of the Banach
 subspace ${\mathcal H}(\mathbb{A})$ of self-adjoint elements.
 Thus the conclusion of Proposition \ref{P:Thomp} applies.
 In particular the conclusion holds for the open cone
 $\mathbb{P}_n$ of positive definite $n\times n$-matrices
 equipped with the Thompson metric.
\end{example}

\section{Means and Barycenters}
We begin this section by recalling several needed notions and results from Section 3 of \cite{LL16}. 

\begin{definition}\label{D:mb}
(1) An \emph{$n$-mean} $G_n$ on a set $X$ for $n\geq 1$ 
is a function $G_n:X^n\to X$ that is idempotent in the sense that $G_n(x,\ldots,x)=x$ for all  $x\in X$. 
\newline
(2) An $n$-mean  $G_n$ is \emph{symmetric} or \emph{permutation invariant} if  for each permuation $\sigma$
of $\{1,\ldots, n\}$, $G_n(\Bx_\sigma) =G_{n}(\Bx)$, where $\Bx=(x_1,\ldots,x_n)$ and 
${\Bx}_{\sigma}=(x_{\sigma(1)},\dots,x_{\sigma(n)})$ . A (symmetric) \emph{mean} $G$ on $X$ is a sequence of   means $\{G_n\}$,  
one (symmetric) mean for each $n\geq 1$.  
\newline (3) A \emph{barycentric
map} or \emph{barycenter} on the finitely supported uniform measures $\mathcal{P}_0(X)$ is a map 
$\beta: \mathcal{P}_0(X)\to X$ satisfying $\beta(\delta_x)=x$ for each $x\in X$.
\end{definition}

For  ${\bf x}=(x_{1},\dots,x_{n})\in X^{n},$ we let
 \begin{equation}\label{E:1}
 {\bf x}^{k}=(x_{1},\dots,x_{n}, x_{1},\dots,x_{n},\dots, x_{1},\dots,x_{n})\in X^{nk},
 \end{equation}  where the number of blocks is $k.$
 We define the \emph{carrier} $S(\mathbf{x})$ of $\Bx$ to be the set of entries in $\mathbf{x}$, i.e., 
 the smallest finite subset $F$ such that $\mathbf{x}\in F^n$.  We set $[\Bx]$ equal to the
equivalence class of all $n$-tuples obtained by permuting the coordinates of $\Bx=(x_1,\dots,x_n)$.  Note that the 
operation $[\Bx]^k=[\Bx^k]$ is well-defined and that all members of $[\Bx]$ all have the same carrier set $S(\Bx)$.  

A tuple $\Bx=(x_1,\ldots,x_n)\in X^n$ \emph{induces} a finitely supported probability measure $\mu$ 
on $S(\Bx)$ by $\mu=\sum_{i=1}^n (1/n) \delta_{x_i}$, where $\delta_{x_i}$ is the point measure of mass 1 at $x_i$.
Since the tuple may contain repetitions of some of its entries, each singleton set  
$\{x\}$ for $x\in\{x_1,\ldots,x_n\}$ will have measure $k/n$, where $k$ is the number
of times that it appears in the listing $x_1,\ldots, x_n$.  Note that every 
member of $[\Bx]$ induces the same finitely supported probability measure.  

\begin{lemma}\label{L:induced}
For each probability measure $\mu$ on $X$ with finite support $F$ for which $\mu(x)(=\mu(\{x\}))$ is rational for each $x\in F$, 
there exists a unique $[\Bx]$ inducing $\mu$ such that any  $[\By]$ inducing $\mu$ is equal to $[\Bx]^k$ for some $k\geq 1$.
\end{lemma}
 
\begin{definition} A mean $G=\{G_{n}\}$ on $X$ is said to be \emph{multiplicative}  if for all  $n,k\geq 2$ and  
all ${\bf x}=(x_{1},\dots,x_{n})\in X^{n}$,
$$G_{n}({\bf  x})=G_{nk}({\bf x}^{k}).$$
If $G$ is also symmetric, then $G$ is called \emph{intrinsic}.
\end{definition}
We have the following corollary to Lemma \ref{L:induced}.
\begin{corollary}\label{C:induced}
Let $G$ be an intrinsic mean.  Then for any finitely supported probability measure $\mu$ with support $F$ and taking on
rational values, we may define $\beta_G(\mu)=G_n(\Bx)$, for any $\Bx\in F^n$ that induces $\mu$.  
\end{corollary}
Corollary \ref{C:induced} provides the basis for the following equivalence.
\begin{proposition}\label{P:meanbary}
There is a one-to-one correspondence between the intrinsic means and the barycentric maps on $X$ given
in one direction by assigning to an intrinsic mean $G$ the barycentric map $\beta_G$ and in the reverse
direction assigning to a barycentric map
$\beta$ the mean $G_n(x_1,\ldots, x_n)=\beta((1/n)\sum_{i=1}^n\delta_{x_i})$.
\end{proposition}

We specialize to means and barycenters in metric spaces.
The following notion of what we call a contractive mean has appeared in other work; see e.g. \cite{LLL}.
\begin{definition} An $n$-mean $G_{n}:X^{n}\to X$ is said to be \emph{contractive} 
if for all ${\bf x}=(x_{1},\dots,x_{n}),{\bf y}=(y_{1},\dots,y_{n})\in X^{n}$
\begin{eqnarray*} d(G_{n}({\bf x}), G_{n}({\bf y}))\leq \frac{1}{n}\sum_{j=1}^{n}d(x_{j},y_{j}).
\end{eqnarray*}
A mean $G=\{G_{n}\}$ is \emph{contractive}  if each $G_{n}$ is contractive.
\end{definition}

We recall the notion of Sturm \cite{St} of a contractive barycentric map on the set of probability measures of
finite first moment on a complete metric space.
\begin{definition} \label{D:cont} A barycentric map $\beta:{\mathcal P}^{1}(X)\to X$ is said to be \emph{contractive} if
$d(\beta(\mu_{1}),\beta(\mu_{2}))\leq d^W(\mu_{1},\mu_{2})$ for all $\mu_{1},\mu_{2}\in{\mathcal P}^{*}(X).$
\end{definition}
A fundamental relationship between contractive intrinsic means and barycentric maps is the  following
(see \cite[Proposition 4.7]{LL16}).
\begin{proposition}\label{P:corres} In a metric space $(X,d)$ the bijective correspondence of Proposition
\ref{P:meanbary} restricts to a bijective correspondence 
between the set of contractive intrinsic means  $G$ on $X$ and the set of contractive
barycentric maps $\beta$ on ${\mathcal P}_0(X)$.  If $X$ is a complete metric space,
then the contractive barycentric maps on ${\mathcal P}_0(X)$ uniquely extend to
contractive barycentric maps on ${\mathcal P}^1(X)$, thus yielding a bijection between 
the set of contractive intrinsic means  $G$ on $X$ and the set of contractive
barycentric maps $\beta$ on ${\mathcal P}^1(X)$.
\end{proposition}

We turn now to connecting these results to our previous results concerning measures on ordered spaces.

\begin{definition}
An $n$-mean $G_n:S^n\to X$ on a partially ordered set $X$ is said to be
\emph{monotonic} if
$G(x_1,\ldots, x_n)\leq G_n(y_1,\ldots, y_n)$ whenever $x_i\leq y_i$ for
$i=1.\ldots, n$.  A barycentric
map $\beta$ defined on $\mathcal{P}_0(X)$ or on$\mathcal{P}^1(X)$ for the
case $X$ is a metric space
is \emph{monotonic} is $\mu\leq \nu$ implies $\beta(\mu)\leq \beta(\nu)$.
\end{definition}

The following lemmas are crucial for connecting the monotonicity of
barycentric maps with the monotonicity of their corresponding  means.

Recall that a bipartite graph is one in which the set of vertices $V$ is
the disjoint union of two sets $A$ and $B$ and each edge connects
some member of $A$ with some member of $B$.  We recall the following
special case of the well-known Hall's Marriage Theorem \cite{Br} from
graph theory.
\begin{lemma}\label{L:Hall}
Let $V=A\cup B$ be the vertices of a bipartite for which
$A=\{a_1,\ldots,a_n\}$ and $B=\{b_1,\ldots b_n\}$ have the same
cardinality.
Suppose for each $C\subseteq A$, we have $\vert C\vert\leq \vert N(C)\vert$,
where $N(C)=\{b\in B: ab$ is an edge for some $a\in C\}$.  Then there
exists a permutation $\sigma$  of $\{1,\ldots, n\}$ such that
$a_ib_{\sigma(i)}$ is an edge for all $i=1,\ldots,n$.
\end{lemma}

\begin{lemma}\label{L:Hall2}
Let $X$ be a partially ordered, let $\mu$ and $\nu$ be uniform
probabilities, each with
support of cardinality $n$.  If $\mu\leq \nu$ in the sense that
$\mu(A)\leq \nu(\ua A)$ for each $A\subseteq$ supp$(\mu)$, then there exists
a permutation $\sigma$ of $\{1,\ldots, n\}$ such that $x_k\leq
y_{\sigma(k)}$ for $1\leq k\leq n$, where $\{x_1,\ldots, x_n\}$
and $\{y_1,\ldots,y_n\}$ are the supports of $\mu$ and $\nu$ resp.
\end{lemma}

\begin{proof}  We define a bipartite graph $G$ with vertices $A\cup B$, where
$A=\{x_1,\ldots,x_n\}$ and $B=\{y_1,\ldots, y_n\}$, except
that we make all the vertices distinct if necessary.
We define $x_iy_j$ to be an edge of the graph if $x_i\leq y_j$ in $X$. 
Let $C$ be a subset of $A$ of cardinality $k$.  Then $\mu(C)=k/n$,
so by hypothesis $\nu(\ua C\cap \mbox{supp}(\nu))\geq k/n$, which
implies $\vert \ua C \cap \mbox{supp}(\nu)\vert\geq k$.  Thus the graph
$G$ satisfies the hypotheses of Lemma \ref{L:Hall}, and the lemma
 follows from that result  and the construction of $G$.

\end{proof}

The next result connects monotonicity for means and barycentric maps.
\begin{proposition}\label{P:zmono}
Let $(M,d,\leq)$ be a metric space equipped with a closed partial order.
Let $G$ be an intrinsic mean on $M$ and let $\beta_G$ be the
corresponding barycentric map on $\mathcal{P}_0(M)$.  Then $G$ is monotonic
if and only if $\beta_G$ is.  
\end{proposition}

\begin{proof} 
Suppose first that $\beta_G$ is monotonic.  Suppose that $x_i\leq y_i$ for $1\leq i\leq n$.
It follows easily that $(1/n)\sum_{i=1}^n\delta_{x_i}\leq (1/n)\sum_{i=1}^n \delta_{y_i}$, so
$$G_n(x_1,\ldots,x_n)=\beta_G\left(\frac{1}{n}\sum_{i=1}^n \delta_{x_i}\right)
\leq \beta_G\left(\frac{1}{n}\sum_{i=1}^n \delta_{y_i}\right)=G_n(y_1,\ldots,y_n).$$
Hence $G$ is monotonic.

Conversely suppose that $G$ is monotonic and suppose that $\mu=(1/k)\sum_{i=1}^k\delta_{x_i}$,
$\nu=(1/m)\sum_{i=1}^m \delta_{y_i}$, and $\mu\leq \nu$.  We set $\Bx=(x_1,\ldots, x_k)$ and $\By=
(y_1,\ldots, y_m)$.   Using the notation of equation (\ref{E:1}), we have 
$\Bx^m,\By^k\in M^n \mbox{ for } n=km$.  We can now rewrite 
$\mu=(1/n)\sum_{i=1}^n \delta_{x_i}$
where the $x_i$ range through the entries of $\Bx^m$.  and similarly we rewrite 
$\nu=(1/n)\sum_{i=1}^n \delta_{y_i}$.  We apply Lemmas \ref{L:stord} and \ref{L:Hall2}
to these alternative representations of $\mu$ and $\nu$ to conclude that 
there exists a permutation $\sigma$ of $\{1,\ldots,n\}$ such that $x_i\leq
y_{\sigma(i)}$ for $1\leq i\leq n$.  Since $G_n$ is monotone and symmetric 
$$\beta_G(\mu)=G_n(x_1,\ldots,x_n)\leq G_n(y_{\sigma(1)},\ldots, y_{\sigma(n)})
=G_n(y_1,\ldots,y_n)=\beta_G(\nu).$$

\end{proof}

\begin{definition}
Let $M$ be a metric space $M$ equipped with a closed partial order.  A pair $(\mu,\nu)\in \mathcal{P}^1(M)\times\mathcal{P}^1(M)$ 
is called \emph{order approximable} if $\mu\leq \nu$ and there exist sequences $\{\mu_n\},\{\nu_n\}\subseteq \mathcal{P}_0(M)$  
such that for each $n$, $\mu_n\leq \nu_n$ and with respect to the Wasserstein metric $\mu_n\to \mu$ and $\nu_n\to \nu$.
\end {definition}

Theorem \ref{T:ordpair} gives important sufficient conditions for each pair $\mu\leq \nu$ in $\mathcal{P}^1(M)$ to be approximable.

\begin{theorem}\label{T:mono}  Let $M$ be an ordered and complete metric space in which each pair $\mu\leq \nu$ in $\mathcal{P}^1(M)$ is approximable.
If $G$ is a monotonic contractive intrinsic mean, then the corresponding barycentric map $\beta_G: \mathcal{P}^1(M)\to M$ is
monotonic.
\end{theorem}

\begin{proof}  The existence of the contractive barycentric map $\beta_G:\mathcal{P}^1(M)\to M$ follows from
Proposition \ref{P:corres}.  That $\beta_G$ is monotonic on $\mathcal{P}_0(M)$ follows from Proposition \ref{P:zmono}.
For $\mu\leq \nu$ in $\mathcal{P}^1(M)$, by hypothesis there exist sequences $\{\mu_n\},\{\nu_n\}\subseteq \mathcal{P}_0(M)$  
such that for each $n$, $\mu_n\leq \nu_n$ and $\mu_n\to \mu$ and $\nu_n\to \nu$.  By continuity of $\beta_G$ and the closedness
of the partial order, it follows that $\beta_G(\mu)\leq\beta_G(\nu)$.

\end{proof}

For the remainder of the section we work in the following setting.  Let ${\mathbb A}$ be a unital $C^{*}$-algebra with
identity $e$. Let ${\mathbb A}^{+}$ be the set of positive
invertible elements of ${\mathbb A}$, a normal open cone of the Banach
 subspace ${\mathcal H}(\mathbb{A})$ of self-adjoint elements.  We suppose further that ${\mathcal H}(\mathbb{A})$
is conditionally directed complete, i.e., that  down-directed subsets of $\mathcal{H}(\mathbb{A})$ that are bounded below have
 an infimum.  
 The \emph{Karcher mean} $\Lambda=\{\Lambda_n\}$ on $\mathbb{A}^+$ is defined as the unique
 solution in $\mathbb{A}^+$ of the Karcher equation
 $$X=\Lambda_n(A_1,\ldots A_n)\Leftrightarrow \sum_{i=1}^n \log(X^{-1/2}A_i X^{-1/2})=0.$$
 It has been shown in \cite{LL14} that this equation does indeed have a unique solution in $\mathbb{A}^+$ and 
 in \cite{LL14,LL16} that the resulting 
 mean $\Lambda_n$ has the following properties:
 \begin{itemize}
 \item[(i)]  $\Lambda_n$ is idempotent and intrinsic, in particular, symmetric:
 \item[(ii)] $($Monotonicity$)$ $B_{i}\leq A_{i}$ for all $1\leq i\leq n \Rightarrow
\Lambda_n(B_1,\ldots, B_n)\leq \Lambda_n(A_1,\ldots,A_n);$
\item[(iii)] (Contractivity) $d(\Lambda_n(A_1,\ldots,A_n),\Lambda_n(B_1,\ldots, B_n))\leq (1/n)\sum_{i=1}^n d(A_i,B_i)$, \\
where $d$ is the Thompson metric.
\end{itemize}

The correspondence of Proposition \ref{P:corres} yields a uniquely determined contractive  barycentric map 
$\beta_\Lambda:\mathcal{P}^1(\mathbb{P})\to \mathbb{P}$ satisfying
$\beta_\Lambda((1/n)\sum_{i=1}^n \delta_{A_i})=\Lambda_n(A_1,\ldots, A_n)$.  Furthermore, in light of the remarks in
Example \ref{E:Cstar} and Theorem \ref{T:mono},  $\beta_\Lambda:\mathcal{P}^1(\mathbb{A}^+)\to\mathbb{A}^+$ 
is monotonic.  We summarize:
\begin{theorem}\label{T:mainmono}
Let $\mathbb{A}$ be a $C^*$-algebra with identity for which the Banach space $\mathcal{H}(A)$ is
conditionally directed complete.  Then the Karcher mean on $\mathbb{A}^+$ extends uniquely to a 
monotonic contractive barycentric map $\beta_\Lambda:\mathcal{P}^1(\mathbb{A}^+)\to \mathbb{A^+}$.
\end{theorem}

\section{Future Work}  Many questions and conjectures remain open.  One would like to show that in the context of this paper
integrable measures on $\mathbb{P}$ satisfy an appropriate Karcher equation and that the solution is unique.  The Karcher mean has 
been shown to satisfy a number of inequalities (e.g., concavity, Ando-Hiai inequality, etc.) and it is of interest to extend, where possible,
these results to the case of Borel probability measures and the Karcher barycentric map.  Indeed the approach and techniques of this 
paper should make possible the extension of quite a number of results involving inequalities of matrix and operator means to 
inequalities of integrable measures.

\end{document}